\documentclass[a4]{amsart}

\input xypic 
\input xy 
\xyoption{all} 
\usepackage{amssymb} 
\usepackage{bbm}
\usepackage{tikz-cd}

\usepackage{float}
\usepackage{graphicx}
\usepackage{accents}

\usepackage[bookmarksnumbered, bookmarksopen,
colorlinks,citecolor=blue,linkcolor=blue,backref]{hyperref}

\newtheorem{theorem}{Theorem}[section]
\newtheorem{proposition}[theorem]{Proposition}
\newtheorem{lemma}[theorem]{Lemma}
\newtheorem{corollary}[theorem]{Corollary}

\theoremstyle{definition}

\newtheorem{remark}[theorem]{Remark}
\newtheorem{example}[theorem]{Example}

\newcounter{RomanNumber}
\newcommand{\MyRoman}[1]{\setcounter{RomanNumber}{#1}\Roman{RomanNumber}}

\newcommand{\conn}{\ensuremath{\#}} 
 
\newcommand{\Gtau}{\ensuremath{\mathcal{G}^{\tau}}} 
\newcommand{\Gzero}{\ensuremath{\mathcal{G}^{0}}} 

\newcounter{bean}

\newcommand{\namedright}[3]{\ensuremath{#1\stackrel{#2}
 {\longrightarrow}#3}}
\newcommand{\nameddright}[5]{\ensuremath{#1\stackrel{#2}
 {\longrightarrow}#3\stackrel{#4}{\longrightarrow}#5}}
\newcommand{\namedddright}[7]{\ensuremath{#1\stackrel{#2}
 {\longrightarrow}#3\stackrel{#4}{\longrightarrow}#5
  \stackrel{#6}{\longrightarrow}#7}}

\newcommand{\larrow}{\relbar\!\!\relbar\!\!\rightarrow}
\newcommand{\llarrow}{\relbar\!\!\relbar\!\!\larrow}
\newcommand{\lllarrow}{\relbar\!\!\relbar\!\!\llarrow}

\newcommand{\llnameddright}[5]{\ensuremath{#1\stackrel{#2}
 {\llarrow}#3\stackrel{#4}{\llarrow}#5}}

\newcommand{\lllnameddright}[5]{\ensuremath{#1\stackrel{#2}
 {\lllarrow}#3\stackrel{#4}{\lllarrow}#5}}

\newcommand{\qqed}{\hfill\Box}

\begin{document}


\title{Homotopy of manifolds stabilized by projective spaces} 

\author{Ruizhi Huang} 
\address{Institute of Mathematics, Academy of Mathematics and Systems Science, 
   Chinese Academy of Sciences, Beijing 100190, China} 
\email{huangrz@amss.ac.cn} 
   \urladdr{https://sites.google.com/site/hrzsea/}
   
\author{Stephen Theriault}
\address{School of Mathematics, University of Southampton, Southampton 
   SO17 1BJ, United Kingdom}
\email{S.D.Theriault@soton.ac.uk}

\date{}


\begin{abstract} 
We study the homotopy of the connected sum of a manifold with a projective space, viewed as a typical way to stabilize manifolds. In particular, we show a loop homotopy decomposition of a manifold after stabilization by a projective space, and provide concrete examples. To do this, we trace the effect in homotopy theory of surgery on certain product manifolds by showing a loop homotopy decomposition after localization away from the order of the image of the classical $J$-homomorphism.
\end{abstract}

\maketitle


\section{Introduction}
The connected sum of a manifold with a projective space is of special interest in geometry and topology. In geometric topology, in analogy to vector bundles and their stabilizations to build up $K$-theory, Kreck \cite{K2} suggested to consider diffeomorphism classes of smooth manifolds modulo connected sum with a prescribed manifold $T$. For $N$ a connected closed $2n$-dimensional smooth manifold, two typical choices of $T$ can be $S^n\times S^n$ and the complex projective space $\mathbb{C}P^n$.
The manifold $N\#T$ is called the {\it $T$-stabilization} of $N$. It is an important and active problem in geometric topology to the classify smooth manifolds up to $T$-stabilizations.  
When $T=S^n\times S^n$, the classification of stable diffeomorphism classes of $2n$-manifolds was systematically studied by Kreck~\cite{K1} using his modified surgery technique. 
When $T=\mathbb{C}P^n$, and $n=2$ in particular, the $\mathbb{C}P^2$-stable classification of smooth $4$-manifolds was deeply investigated recently by Kasprowski, Powell and Teichner \cite{KPT} based on \cite{K1}.

In addition to the geometric topology perspective, there is an analytic way to look at manifolds stabilized by projective spaces. 
When $N$ is a complex manifold of complex dimension $n$, the $\mathbb{C}P^n$-stabilization of $N$ can be obtained by blowing up at a point of $N$. Similarly, when $n$ is even and $N$ is a regular quaternionic manifold of quaternionic dimension $n/2$, Gentili, Gori and Sarfatti \cite{GGS} showed that the connected sum of $N$ with the quaternionic projective space $\mathbb{H}P^{\frac{n}{2}}$ is quaternionic diffeomorphic to the blow-up of $N$ at a point.

In this paper, we study the homotopy theory of manifolds after stabilization by a projective 
space $\mathbb{C}P^n$ or $\mathbb{H}P^{\frac{n}{2}}$. 
At the heart of the paper is an investigation of the effect in homotopy theory of surgery on certain manifolds. 
Specifically, we consider the case of surgery on a product manifold $N\times S^{k-1}$ along the canonical 
embedding of $S^{k-1}$ for $k\geq 2$. This involves a novel combination of ideas and techniques from 
geometric topology and homotopy theory. To the best of our knowledge 
the only surgery that has been studied from a homotopy theoretic point of view is a connected sum. 
In that case there are numerous rational results, for example, \cite{FHT,HaL,L}, but integrally or $p$-locally 
there are only a few very recent results in \cite[Theorem~1.4]{T}, \cite{JS}, \cite{C}. 

To describe the effect in homotopy of this surgery, let $N$ be an $n$-dimensional connected closed manifold 
and let $N_{0}$ be $N$ with a small $n$-dimensional disc removed. For $k\geq 2$, let 
$\tau: S^{k-1}\rightarrow O(n)$ be a map. 
Using the standard linear action of $O(n)$ on $S^{n-1}$, define the map 
\[
f\colon\namedright{S^{n-1}\times S^{k-1}}{}{S^{n-1}\times S^{k-1}} 
\]
by $f(a, t)=(\tau(t)a, t)$. Let $\iota: S^{k-1}\rightarrow D^{k}$ be the standard 
inclusion. Let 
\begin{equation}\label{gyrationdef}
\Gtau(N):= 
    (S^{n-1}\times D^{k})\cup_f (N_0\times S^{k-1}).
\end{equation}
be the space obtained by gluing together the image of 
${\rm id}\times \iota$ in the left copy of $S^{n-1}\times D^{k}$ and the image of $(i\times {\rm id})\circ f$ in the right copy of $N_0\times S^{k-1}$. 
This construction generalizes the \emph{gyration} construction of Gonz\'{a}lez-Acu\~{n}a~\cite{GA} and the \emph{suspension} construction of Duan \cite{D}.
It is especially useful in study of toric topology \cite{GLdM} and regular circle actions on manifolds \cite{D}.

Notice that the disc $D^{n}$ contained in $N$ that is removed to obtain $N_{0}$ 
results in an embedding 
\(\namedright{D^{n}\times S^{k-1}}{}{N\times S^{k-1}}\).  
Therefore~$\Gtau(N)$ is obtained from $N\times S^{k-1}$ by a surgery removing 
the interior of the embedding and gluing in $S^{n-1}\times D^{k}$. 
In particular, $\Gtau(N)$ is a closed $(n+k-1)$-dimensional manifold. 
When $\tau$ is trivial, or equivalently $f$ is the identity map, we write $\tau=0$ and denote the manifold $\Gtau(N)$ by $\Gzero(N)$. 

Theorem~\ref{gyrationtypeintro} gives a homotopy decomposition of the based loops on $\Gtau(N)$ after 
localization away from a small set $\mathcal{P}_{k}$ of primes that depends on the image of the classical 
$J$-homomorphism (and is made explicit in Section~\ref{sec: gyration}). For any $CW$-complex $X$, let $\Omega X$ be the based loop space of $X$, and $\Sigma^i X$ be the $i$-fold suspension of $X$.
If $N$ is a closed manifold, notice that $N_0$ has spherical boundary. 

\begin{theorem} 
   \label{gyrationtypeintro} 
   Let $N$ be a connected closed $n$-dimensional manifold and $n\geq k+2\geq 4$. After localization away from $\mathcal{P}_k$, there is a homotopy fibration 
   \[\nameddright{\Sigma^k F}{}{\Gtau(N)}{t}{N_0}\] 
   where $F$ is the homotopy fibre of the inclusion
   \(i: \namedright{S^{n-1}}{}{N_0}\) of the boundary. Further, this homotopy fibration splits after looping to give a homotopy equivalence away from $\mathcal{P}_k$
   \[\Omega\Gtau(N)\simeq\Omega N_0\times\Omega\Sigma^k F.\] 
In the special case when $\tau$ is trivial, the statement holds without localization, and there is an integral homotopy equivalence
   \[\Omega\Gzero(N)\simeq\Omega N_0\times\Omega\Sigma^k F.\]  
\end{theorem} 

Theorem~\ref{gyrationtypeintro} is stated for manifolds due to the connection to surgery, but the definition 
of $\Gtau(N)$ and the argument proving Theorem~\ref{gyrationtypeintro} hold in the more general case 
of $n$-dimensional Poincar\'{e} Duality complexes and even 
for $n$-dimensional $CW$-complexes $N$ with $H^{n}(N)\cong\mathbb{Z}$ and $N_{0}$ as the $(n-1)$-skeleton.
It is not clear whether the localization hypothesis in the case when $\tau$ is nontrivial is necessary. 
For example, when $k=2$ there are two possible homotopy classes for $\tau$, a trivial one denoted $0$ 
and a nontrivial one denoted $\tau$, and Duan~\cite[Example 3.4]{D} shows that 
$\Gtau(\mathbb{H}P^{2})$ and $\Gzero(\mathbb{H}P^{2})$ represent different elements in the 
$9$-dimensional Spin bordism group. However, it may be the case that after looping the homotopy 
types align. In general, it would be interesting to know to what extent the localization hypothesis can be removed. 

Theorem~\ref{gyrationtypeintro} is the key ingredient in determining the homotopy types of manifolds 
stabilized by projective spaces. It should be noted that some of the other ingredients use additional 
geometric input such as the Whitney embedding theorem. Therefore, to avoid possible technicalities in 
achieving maximal generalization, we restrict to the case of closed manifolds.

\begin{theorem}\label{stabledecthm}
Let $N$ be a connected closed $2n$-dimensional smooth manifold. Let $F$ be the homotopy fibre of the inclusion $S^{2n-1}\hookrightarrow N_0$ of the boundary. The following hold: 
\begin{itemize} 
\item if $n\geq 2$ is even, there is a homotopy equivalence 
\[
\Omega (N\# \mathbb{C}P^n)\simeq S^{1}\times \Omega N_0\times\Omega\Sigma^2 F;
\] 
\item if $n\geq 2$ is odd, there is a homotopy equivalence after localization away from $2$
\[
\Omega (N\# \mathbb{C}P^n)\simeq S^{1}\times \Omega N_0\times\Omega\Sigma^2 F;
\]
\item if $n\geq 4$ is even, there is a homotopy equivalence after localization away from $2$ and $3$
\[
\Omega (N\# \mathbb{H}P^{\frac{n}{2}})\simeq S^{3}\times \Omega N_0\times\Omega\Sigma^4 F.
\]
\end{itemize}
\end{theorem} 


Theorem \ref{stabledecthm} describes the homotopy type of $\Omega (N\# \mathbb{C}P^n)$ and $\Omega (N\# \mathbb{H}P^{\frac{n}{2}})$ in terms of the internal structure of $N$. 
In particular, it allows us to compute the homotopy groups of~$N$ after stabilization in terms of those of 
spheres, the deleted manifold $N_0$ and iterated suspensions of the homotopy fibre~$F$. 
In general, the homotopy theory of $N_0$ and $F$ could be complicated. However, they are 
accessible in interesting cases as illustrated by the following example. 
\begin{example}\label{ex1intro}
Let $\mathbb{O}P^2$ be the octonionic projective plane.
For $n\geq 2$, there are homotopy equivalences 
\[
\begin{split}
\Omega (\mathbb{C}P^{2n}\#\mathbb{C}P^{2n})&\simeq S^1\times S^1\times \Omega S^3\times \Omega S^{4n-1},\\
\Omega (\mathbb{C}P^{2n}\#\mathbb{H}P^{n})&\simeq S^1\times S^3\times \Omega S^5\times \Omega S^{4n-1}, \\
\Omega (\mathbb{C}P^{8}\#\mathbb{O}P^{2}) &\simeq S^1\times S^7\times \Omega S^9\times \Omega S^{15}.
\end{split}
\]
Further after localization away from $2$ there is a homotopy equivalence 
\[
\Omega (\mathbb{C}P^{2n+1}\#\mathbb{C}P^{2n+1})\simeq S^1\times S^1\times \Omega S^3\times \Omega S^{4n+1},
\]
and after localization away from $2$ and $3$ there are homotopy equivalences 
\[
\begin{split}
\Omega (\mathbb{H}P^{n}\#\mathbb{H}P^{n})&\simeq S^3\times S^3\times \Omega S^7\times \Omega S^{4n-1},\\
\Omega (\mathbb{H}P^{4}\#\mathbb{O}P^{2})&\simeq S^3\times S^7\times \Omega S^{11}\times \Omega S^{15}.\\
\end{split}
\]
This example is shown in Proposition \ref{Pnsumprop}.

Recently, Duan \cite[Example 3.5]{D} has shown the corresponding results for $\mathbb{C}P^{2n}\#\mathbb{C}P^{2n}$ and $\mathbb{C}P^{2n}\#\mathbb{H}P^{n}$ by a more geometric argument. He treated only the $\mathbb{C}P^{2n}$ case because of a requirement for a non-spin condition in \cite[Theorem B]{D}; also see Lemma \ref{duanthmb}.
\end{example} 

An interesting refinement of Theorem~\ref{stabledecthm} occurs when the inclusion 
\(\namedright{N_{0}}{h}{N}\) 
is such that $\Omega h$ has a right homotopy inverse. By~\cite[Proposition 3.5]{BT2} this implies that 
there is a homotopy equivalence 
$\Omega N_{0}\simeq\Omega N\times\Omega(\Sigma^{2n-1}\Omega N\vee S^{2n-1})$. 
Substituting this into the homotopy equivalences in Theorem~\ref{stabledecthm} lets 
us directly compare the local homotopy types of $\Omega N$ and, for example, 
$\Omega(N\conn\mathbb{C}P^{n})$. There are several interesting families of manifolds 
with the property that $\Omega h$ has a right homotopy inverse, including $(n-1)$-connected 
$2n$-dimensional manifolds provided $n\notin\{ 2, 4,8\}$~\cite{BT1}. For this family, in 
Section~\ref{sec: ex2} we go even further and give a much finer decompositions as 
compared to Theorem~\ref{stabledecthm}. 

In the rational case, with mild additional hypotheses, the map $\Omega h$ always has a right 
homotopy inverse. For a rational space~$X$, 
the cohomology ring $H^{\ast}(X;\mathbb{Q})$ is \emph{monogenic} if it is generated by one element. 
For example, $H^{\ast}(\mathbb{C}P^{n};\mathbb{Q})$ and $H^{\ast}(\mathbb{H}P^{n};\mathbb{Q})$ are all monogenic. 

\begin{corollary} 
   \label{rationalcase} 
   Let $N$ be a simply-connected, closed $2n$-dimensional smooth manifold whose rational cohomology 
   is not monogenic. Let $F$ be the homotopy fibre of the inclusion 
   \(S^{2n-1}\hookrightarrow N_{0}\) 
   of the boundary. Then there are rational homotopy equivalences: 
   \begin{itemize} 
      \item $\Omega(N\conn\mathbb{C}P^{n})\simeq S^{1}\times  \Omega N\times 
                     \Omega(\Sigma^{2n-1}\Omega N\vee S^{2n-1})\times\Omega\Sigma^{2} F$ if $n\geq 2$; 
      \item $\Omega(N\conn\mathbb{H}P^{\frac{n}{2}})\simeq S^{3}\times  \Omega N\times 
                  \Omega(\Sigma^{2n-1}\Omega N\vee S^{2n-1})\times\Omega\Sigma^{4} F$ if $n\geq 4$ is even.            
   \end{itemize} 
\end{corollary} 

Using $\Omega(N\conn\mathbb{C}P^{n})$ as an example, Corollary~\ref{rationalcase} 
states that $\Omega N$ is a retract of $\Omega(N\conn\mathbb{C}P^{n})$ and it explicitly identifies 
the complement. This significantly enhances a result of Halperin and 
Lemaire~\cite[Th\'{e}or\`{e}me 5.4 (ii)]{HaL} in the case of a connected sum with $\mathbb{C}P^{n}$, 
which would only have said that $\pi_{\ast}(N\conn\mathbb{C}P^{n})$ contains a free Lie algebra 
on two generators. This Lie algebra can be seen is as follows. Any suspension is rationally 
homotopy equivalent to a wedge of spheres, so $\Sigma^{2n-1}\Omega N$ is rationally homotopy equivalent 
to a wedge of at least one sphere. Therefore $\Sigma^{2n-1}\Omega N\vee S^{2n-1}$ is rationally homotopy 
equivalent to a wedge of at least two spheres. This wedge of two spheres generates a Lie algebra 
on two generators in $\pi_{\ast}(N\conn\mathbb{C}P^{n})$.


It is worthwhile to compare Theorem \ref{stabledecthm} and Example \ref{ex1intro} with the based loop space decomposition theorem of the second author in \cite[Theorem 1.4]{T}. 
Let $X\ltimes Y$ be the {\it left half-smash} of $X$ and~$Y$, defined as the quotient space $(X\times Y)/(X\times \ast)$.
In our context, Theorem 1.4 of \cite{T} implies that if the obvious inclusion $h: N_0\hookrightarrow N$ has the property that $\Omega h$ admits a right homotopy inverse, then there is a homotopy equivalence
\begin{equation}\label{Tdeceq}
\Omega (N\# \mathbb{F}P^m)\simeq \Omega N \times \Omega (\Omega N\ltimes \mathbb{F}P^{m-1}),
\end{equation}
where $\mathbb{F}P^m=\mathbb{C}P^n$ or $\mathbb{H}P^{\frac{n}{2}}$ in 
appropriate dimensions. The decomposition (\ref{Tdeceq}) is different from ours. Critically, (\ref{Tdeceq}) 
needs the extra hypothesis of a right homotopy inverse, while Theorem \ref{stabledecthm} does not. 
In particular, Example \ref{ex1intro} cannot be obtained by \cite[Theorem~1.4]{T}. 
Even when a right homotopy inverse exists, the factors on the right side of (\ref{Tdeceq}) 
are formulated differently, involving~$N$ and a lower skeleton of $\mathbb{F}P^{m}$, whereas 
the factors on the right side of the decompositions in Theorem~\ref{stabledecthm} involve a sphere 
and spaces related to $N_{0}$. 

The paper is organized as follows. Theorem \ref{gyrationtypeintro} is proved in Section \ref{sec: gyration}. 
In Section \ref{sec: sum}, following~\cite{D} we use an $S^{k-1}$-bundle to relate the connected sum $N\#\mathbb{F}P^{m}$ with $\Gtau(N)$, and then reduce the proof of Theorem \ref{stabledecthm} to the decomposition of the loops on $\Gtau(N)$. 
In Section~\ref{sec: ex1} we specialize to prove Theorem \ref{stabledecthm}, establish Example \ref{ex1intro} 
and prove Corollary~\ref{rationalcase}. Section \ref{sec: fibattach} is devoted to further studying the homotopy type of the homotopy fibre~$F$ under an additional hypothesis, and this is used in Section~\ref{sec: ex2} to give more examples.

\medskip

\noindent{\bf Acknowledgement} 
Ruizhi Huang was supported in part by the National Natural Science Foundation of China (Grant nos. 11801544 and 12288201), the National Key R\&D Program of China (No. 2021YFA1002300), the Youth Innovation Promotion Association of Chinese Academy Sciences, and the ``Chen Jingrun'' Future Star Program of AMSS.

The authors are grateful to Prof. Haibao Duan for helpful discussions regarding his suspension operations on manifolds, and to the referees for many valuable comments and suggestions. 

\section{A loop decomposition of certain product manifolds after surgery}
\label{sec: gyration}
If $W$ is a manifold, let~$\partial W$ be its boundary. 
\emph{Surgery} is a process of producing one manifold from another based on the observation that 
\[\partial(D^{n}\times S^{k-1})=S^{n-1}\times S^{k-1}=\partial(S^{n-1}\times D^{k}).\] 
Given an embedding of $D^{n}\times S^{k-1}$ in a manifold $M$, remove the interior of $D^{n}\times S^{k-1}$ 
from $M$, and along the boundary $S^{n-1}\times S^{k-1}$ glue in $S^{n-1}\times D^{k}$. A variation 
is to do the gluing not via the identity map on $S^{n-1}\times S^{k-1}$ but via a self diffeomorphism 
of $S^{n-1}\times S^{k-1}$ over $S^{k-1}$. The choice of diffeomorphism is a \emph{framing} of the surgery.

In our case, we will perform a surgery on $N\times S^{k-1}$ where $N$ is a connected closed $n$-dimensional manifold. 
Let $N_0=N-D^{n}$ be $N$ with a small $n$-dimensional disc removed. Hence, $\partial N_0=S^{n-1}$. 
Also, up to homotopy equivalence, $N_0$ is the $(n-1)$-skeleton of $N$ and there is 
a homotopy cofibration 
\[\nameddright{S^{n-1}}{i}{N_0}{}{N}\] 
that attaches the top cell of $N$, where the map $i$ represents the inclusion of the boundary 
into $N_0$. 

The source of framings comes from the action 
\(\namedright{O(n)\times\mathbb{R}^{n}}{}{\mathbb{R}^{n}}\) 
of the orthogonal group $O(n)$ on the vector space $\mathbb{R}^{n}$ given by applying matrix 
multiplication. As this preserves norms, it restricts to an action 
\(\namedright{O(n)\times S^{n-1}}{}{S^{n-1}}\).
As in the introduction, for $k\geq 2$, let $\tau: S^{k-1}\rightarrow O(n)$ be a map. 
Using the action of $O(n)$ on $S^{n-1}$, define the map 
\[
f\colon\namedright{S^{n-1}\times S^{k-1}}{}{S^{n-1}\times S^{k-1}} 
\]
by $f(a, t)=(\tau(t)a, t)$. 

Let $\iota: S^{k-1}\rightarrow D^{k}$ be the standard inclusion. Let 
\begin{equation}\label{gyrationdef}
\Gtau(N):= 
    (S^{n-1}\times D^{k})\cup_f (N_0\times S^{k-1}).
\end{equation}
be the space obtained by gluing together the image of 
${\rm id}\times \iota$ in the left copy of $S^{n-1}\times D^{k}$ and the image of $(i\times {\rm id})\circ f$ in the right copy of $N_0\times S^{k-1}$. Equivalently, $\Gtau(N)$ is the pushout 
\begin{equation} 
  \label{gyrationpo} 
  \diagram 
      S^{n-1}\times S^{k-1}\rto^-{{\rm id}\times \iota}\dto^{(i\times {\rm id})\circ f} &  S^{n-1}\times D^{k}\dto \\ 
      N_0\times S^{k-1}\rto & \Gtau(N). 
  \enddiagram 
\end{equation} 
Notice that the disc $D^{n}$ contained in $N$ that is removed to obtain $N_{0}$ results 
in an embedding 
\(\namedright{D^{n}\times S^{k-1}}{}{N\times S^{k-1}}\).  
Therefore~$\Gtau(N)$ is obtained from $N\times S^{k-1}$ by a surgery removing 
the interior of the embedding and gluing in $S^{n-1}\times D^{k}$ (commonly referred to as a 
surgery along $S^{k-1}$). The framing is determined by $\tau$.
In particular, $\Gtau(N)$ is a closed $(n+k-1)$-dimensional manifold. 
When $\tau$ is trivial, or equivalently $f$ is the identity map, we write $\tau=0$ and denote the manifold $\Gtau(N)$ by $\Gzero(N)$. 

In the special case when $k=2$ and $\tau$ is trivial, Gonz\'{a}lez-Acu\~{n}a~\cite{GA} refer to $\mathcal{G}^{0}(N)$ as a \emph{gyration}, an object further studied in the context of toric topology  
by Gitler and L\'{o}pez de Medrano~\cite{GLdM}. Also, when $k=2$, Duan \cite{D} refers to $\mathcal{G}^{\tau}(N)$ as a \emph{suspension} and uses it to study regular circle actions on manifolds based on the work of Goldstein and Lininger \cite{GL} on $6$-manifolds.

The goal of this section is to prove the homotopy decomposition of $\Omega\Gtau(N)$ 
in Theorem~\ref{gyrationtypeintro}. This requires several steps along the way. One ingredient is the classical $J$-homomorphism
\[
J: \pi_{k-1} (O(n))\rightarrow \pi_{n+k-1}(S^n). 
\] 
This is defined as follows. Represent an element in $\pi_{k-1}(S^{n-1})$ by a map 
\(\tau\colon\namedright{S^{k-1}}{}{O(n)}\).  
Using the action 
\(\theta\colon\namedright{O(n)\times S^{n-1}}{}{S^{n-1}}\) 
we obtain a composite 
\begin{equation} 
  \label{Jhom} 
  \nameddright{S^{k-1}\times S^{n-1}}{\tau\times 1}{O(n)\times S^{n-1}}{\theta}{S^{n-1}}. 
\end{equation}  
In general, the standard quotient map 
\(\namedright{A\times B}{}{A\wedge B}\) 
has the property that its suspension has a right homotopy inverse 
\(\mathfrak{t}\colon\namedright{\Sigma A\wedge B}{}{\Sigma(A\times B)}\) 
which can be chosen such that $\Sigma\pi_{1}\circ\mathfrak{t}$ and $\Sigma\pi_{2}\circ\mathfrak{t}$ are null homotopic, 
where $\pi_{1}$ and $\pi_{2}$ are the projections onto the first and second factor respectively. 
In our case, suspending and precomposing~(\ref{Jhom}) with $\mathfrak{t}$ gives a map 
\(\tau'\colon\namedright{S^{n+k-1}}{}{S^{n}}\) 
representing a class in $\pi_{n+k-1}(S^{n})$. Define $J([\tau])=[\tau']$. It is a standard fact that~$J$ 
is well-defined and is a group homomorphism.

In the stable range, that is, when $n\geq k+1$, famous work of Adams \cite{A} and Quillen \cite{Q} shows that the image of $J$ is
\begin{equation}\label{imJeq}
{\rm Im}\, J\cong \left\{\begin{array}{cc}
0 & k\equiv 3, 5, 6, 7 \ {\rm mod} \ 8, \\
\mathbb{Z}/2 & k\equiv 1, 2\ {\rm mod} \ 8, k\neq 1, \\
\mathbb{Z}/d_s& k=4s,
\end{array}\right.
\end{equation}
where $d_s$ is the denominator of $B_s/4s$ and $B_s$ is the $s$-th Bernoulli number defined by 
\[\frac{z}{e^z-1}=1-\frac{1}{2}z-\sum\limits_{s\geq 1} B_s \frac{z^{2s}}{(2s)!}.\] 
For each $k\geq 2$, let $\mathcal{P}_{k}$ be the set of prime numbers such that
\begin{equation}\label{pkdefeq}
\mathcal{P}_{k}= \left\{\begin{array}{cc}
\emptyset & k\equiv 3, 5, 6, 7 \ {\rm mod} \ 8, \\
\{2\} & k\equiv 1, 2\ {\rm mod} \ 8, k\neq 1, \\
\{p~|~p~{\rm divides}~d_s\}& k=4s,
\end{array}\right.
\end{equation}

\begin{lemma}
Let $n\geq k+2\geq 4$. After localization away from $\mathcal{P}_k$, there exists a map $s: S^{n-1}\times D^{k}\rightarrow N_0$ such that the diagram
\begin{equation} 
  \label{sdiag} 
  \diagram 
      S^{n-1}\times S^{k-1}\rto^-{{\rm id}\times \iota}\dto^{(i\times {\rm id})\circ f} &  S^{n-1}\times D^{k}\dto^{s} \\ 
    N_0\times S^{k-1}\rto^-{\pi_1} & N_0
  \enddiagram 
\end{equation} 
homotopy commutes, where $\pi_{1}$ is the projection onto the first factor. 

Moreover, when $f$ is the identity map, $s$ can be chosen so that Diagram (\ref{sdiag}) homotopy commutes without localization. 
\end{lemma}
\begin{proof}
By the definition of $f$, $(\pi_1\circ (i\times {\rm id})\circ f)(a, t)=i(\tau(t)a)$. Thus 
the composition $\pi_1\circ (i\times {\rm id})\circ f$ is the same as the composition 
\[\gamma\colon\namedddright{S^{n-1}\times S^{k-1}}{1\times\tau}{S^{n-1}\times O(n)}{\theta}{S^{n-1}}{i}{N_{0}}.\] 
We therefore want to show that $\gamma$ extends to a map 
\(s\colon\namedright{S^{n-1}\times D^{k}}{}{N_{0}}\). 
To do this, we will show that $\theta\circ(1\times\tau)$ extends to a map 
\(s'\colon\namedright{S^{n-1}\times D^{k}}{}{S^{n-1}}\) 
and set $s=i\circ s'$. 

Since $D^{k}$ is contractible, showing that $\theta\circ(1\times\tau)$ extends to $s'$ is equivalent 
to showing that there is a homotopy commutative diagram 
\begin{equation} 
\begin{aligned}
  \label{gammadgrm} 
  \xymatrix{ 
     S^{n-1}\times S^{k-1}\ar[r]^-{\pi_{1}}\ar[d]^{1\times\tau} & S^{n-1}\ar[d]^{\gamma} \\ 
     S^{n-1}\times O(n)\ar[r]^-{\theta} & S^{n-1} }
     \end{aligned}
\end{equation}  
for some map $\gamma$. By hypothesis, $n\geq k+1$, so $\pi_{k-1}(S^{n-1})\cong 0$, implying 
that the restriction of $\theta\circ(1\times\tau)$ to $S^{k-1}$ is null homotopic. Thus there is 
a homotopy commutative diagram 
\begin{equation} 
\begin{aligned}
  \label{deltadgrm} 
  \xymatrix{ 
     S^{n-1}\times S^{k-1}\ar[r]^-{q}\ar[d]^{1\times\tau} & S^{n-1}\rtimes S^{k-1}\ar[d]^{\delta} \\ 
     S^{n-1}\times O(n)\ar[r]^-{\theta} & S^{n-1} }
     \end{aligned}
\end{equation} 
for some map $\delta$, where $S^{n-1}\rtimes S^{k-1}$ is the quotient space $(S^{n-1}\times S^{k-1})/S^{k-1}$ 
and $q$ is the quotient map. It is well known that there is a homotopy equivalence 
$S^{n-1}\rtimes S^{k-1}\simeq S^{n-1}\vee (S^{n-1}\wedge S^{k-1})$. Let $\delta_{1}$ and $\delta_{2}$ 
be the restrictions of $\delta$ to $S^{n-1}$ and $S^{n-1}\wedge S^{k-1}$ respectively. 
Since the restriction of both $\theta\circ(1\times\tau)$ and $q$ to $S^{n-1}$ is the identity map, 
we obtain that $\delta_{1}$ is the identity map. Thus if~$\delta_{2}$ is null homotopic then $\delta\circ q$ 
is homotopic to $\pi_{1}$ and we may take $\gamma$ to be the identity map in~(\ref{gammadgrm}). 

It remains to show that $\delta_{2}$ is null homotopic. As noted after~(\ref{Jhom}), the map 
\(\namedright{\Sigma S^{n-1}\wedge S^{k-1}}{\mathfrak{t}}{\Sigma(S^{n-1}\times S^{k-1})}\) 
composes trivially with the suspension of the projection to $S^{n-1}$ and is the identity map when 
pinched to $\Sigma S^{n-1}\wedge S^{k-1}$. 
Thus precomposing the suspension of~(\ref{deltadgrm}) with $\mathfrak{t}$ shows 
that $\Sigma\delta_{2}$ is homotopic to the composite 
\(\lllnameddright{\Sigma S^{n-1}\wedge S^{k-1}}{\mathfrak{t}}{\Sigma(S^{n-1}\times S^{k-1})}
      {\Sigma(\theta\circ(1\times\tau))}{S^{n}}\). 
But this is the definition of $J([\tau])$. 
If $f$ is the identity map, then $\tau$ is trivial and so $\Sigma\delta_{2}$ 
is null homotopic.
Otherwise, as $k\geq 2$ and we have localized away 
from $\mathcal{P}_{k}$, the image of the $J$-homomorphism is trivial. Hence in both cases $\Sigma\delta_{2}$ 
is null homotopic. Since $n\geq k+2$, the Freudenthal suspension theorem implies that the 
suspension map 
\(\namedright{\pi_{n+k-2}(S^{n-1})}{}{\pi_{n+k-1}(S^{n})}\) 
is an isomorphism. Thus $\delta_{2}$ is null homotopic. 
\end{proof} 

\begin{lemma} 
   \label{CQ} 
   Let $n\geq k+2\geq 4$. After localization away from $\mathcal{P}_k$, 
   there is a map 
   \(t\colon\namedright{\Gtau(N)}{}{N_0}\) 
   such that the composite 
   \(\nameddright{N_0\times S^{k-1}}{}{\Gtau(N)}{t}{N_0}\) 
   is homotopic to the projection onto the first factor and the composite 
   \(\nameddright{S^{n-1}\times D^k}{}{\Gtau(N)}{t}{N_0}\) 
   is homotopic to 
   \(\namedright{S^{n-1}\times D^k}{s}{N_0}\). 
   
Moreover, when $f$ is the identity map, the statement holds without localization.
\end{lemma} 

\begin{proof} 
Consider the diagram 
\[\xymatrix{ 
   S^{n-1}\times S^{k-1}\ar[r]^{{\rm id}\times \iota}\ar[d]^{(i\times {\rm id})\circ f} & S^{n-1}\times D^k\ar[d]\ar@/^/[ddr]^{s} &  \\ 
   N_0\times S^{k-1}\ar[r]\ar@/_/[drr]_{\pi_{1}} & \Gtau(N) \ar@{.>}[dr]^(0.4){t} & \\ 
   & & N_0 }  
\]
where the inner square is a pushout by Diagram (\ref{gyrationpo}) and the outer square homotopy 
commutes by Diagram~(\ref{sdiag}). Since ${\rm id}\times \iota$ is obviously a cofibration, the inner pushout is a homotopy pushout. Hence there is a map $t$ that makes the two triangular regions homotopy commute. 
\end{proof} 

For the remainder of the section, we always suppose that $n\geq k+2\geq 4$ and work in the homotopy category after localization away from $\mathcal{P}_k$.
Consequently, if  
\(i_{1}\colon\namedright{N_0}{}{N_0\times S^{k-1}}\) 
is the inclusion of the first factor, then by Lemma~\ref{CQ} the composite 
\(\namedddright{N_0}{i_{1}}{N_0\times S^{k-1}}{}{\Gtau(N)}{t}{N_0}\) 
is homotopic to $\pi_{1}\circ i_{1}$, which is homotopic to the identity map on $N_0$. 
This proves the following. 

\begin{lemma} 
   \label{tinverse} 
   The map 
   \(\namedright{\Gtau(N)}{t}{N_0}\) 
   has a right homotopy inverse.~$\qqed$ 
\end{lemma} 
 
We now identify some homotopy fibres associated to the maps in Lemma~\ref{CQ}. 
Define the spaces~$F$ and $H$ and the map $g$ by the homotopy fibrations  
\begin{equation} 
\label{N0cube1} 
\begin{split} 
\nameddright{F}{g}{S^{n-1}}{i}{N_{0}} \\ 
\nameddright{H}{}{\Gtau(N)}{t}{N_{0}}.  
\end{split} 
\end{equation} 
If 
\(i_{1}\colon\namedright{S^{n-1}}{}{S^{n-1}\times D^{k}}\) 
is the inclusion of the first factor, then as $s\circ i_{1}\simeq i$ there is a homotopy fibration diagram 
\[\diagram 
     F\rto^-{g}\dto^{f'} & S^{n-1}\rto^-{i}\dto^{i_{1}} & N_{0}\ddouble \\ 
     F'\rto & S^{n-1}\times D^{k}\rto^-{s} & N_{0}.  
  \enddiagram\] 
that defines the space $F'$ and the map $f'$.   
Since $i_{1}$ is a homotopy equivalence, the five-lemma applied to the long exact sequence of 
homotopy groups implies that $f'$ induces an isomorphism on homotopy groups and so is a 
homotopy equivalence since all spaces are assumed to have the homotopy type of $CW$-complexes. 
Thus there is a homotopy fibration 
\begin{equation} 
\label{N0cube2} 
\nameddright{F}{g'}{S^{n-1}\times D^{k}}{s}{N_{0}} 
\end{equation} 
where $g'=i_{1}\circ g$. Next, define the space $F''$ and the map $f''$ by the homotopy fibration diagram 
\[\diagram 
       F''\rto\dto^{f''} & S^{n-1}\times S^{k-1}\rto^-{s\circ(id\times\iota)}\dto^{id\times\iota} 
           & N_{0}\ddouble \\ 
       F\rto^-{g'} & S^{n-1}\times D^{k}\rto^-{s} & N_{0}. 
  \enddiagram\] 
This fibration diagram implies that $F''$ is the homotopy pullback of $g'$ and $id\times\iota$. 
Since $g'=i_{1}\circ g$, there is an iterated homotopy pullback diagram 
\begin{equation} 
  \label{N0cubea} 
  \diagram 
      F\times S^{k-1}\rto^-{g\times id}\dto^{\pi_{1}} & S^{n-1}\times S^{k-1}\rto^-{id\times id}\dto^{\pi_{1}} 
           & S^{n-1}\times S^{k-1}\dto^{id\times\iota} \\ 
      F\rto^-{g} & S^{n-1}\rto^-{i_{1}} & S^{n-1}\times D^{k}. 
  \enddiagram 
\end{equation}  
Here, the right square is a homotopy pullback since $D^{k}$ is contractible, and the left square 
is a homotopy pullback by the naturality of the projection $\pi_{1}$. Since the outer rectangle 
is the homotopy pullback of $g^\prime=i_{1}\circ g$ and $id\times\iota$, we see that $F''\simeq F\times S^{k-1}$ 
and $f''\simeq\pi_{1}$. Thus there is a homotopy fibration 
\begin{equation} 
\label{N0cube3} 
\llnameddright{F\times S^{k-1}}{g\times id}{S^{n-1}\times S^{k-1}}{s\circ(id\times\iota)}{N_{0}}. 
\end{equation} 

Therefore, composing each of the four corners of~(\ref{gyrationpo}) with the map 
\(\namedright{\Gtau(N)}{t}{N_0}\) 
and taking homotopy fibres, from~(\ref{N0cube1}), (\ref{N0cube2}) and~(\ref{N0cube3}) 
we obtain a homotopy commutative cube 
\begin{equation} 
  \label{Qcube} 
  \spreaddiagramcolumns{-1pc}\spreaddiagramrows{-1pc} 
   \diagram
      F\times S^{k-1}\rrto^-{a}\drto^-{b}\ddto^-(0.33){g\times id} & & F\dline^-{g'}\drto & \\
      & S^{k-1}\rrto\ddto^(0.25){i_{2}} & \dto & H\ddto \\
      S^{n-1}\times S^{k-1}\rline^(0.6){{\rm id}\times \iota}\drto^(0.6){(i\times {\rm id})\circ f} & \rto & S^{n-1}\times D^k\drto & \\
      & N_0\times S^{k-1}\rrto & & \Gtau(N)
  \enddiagram 
\end{equation}  
in which the bottom face is a homotopy pushout and the four sides are homotopy 
pullbacks, and the maps $a$ and $b$ are induced maps of fibres. Mather's Cube 
Lemma~\cite{M} implies that the top face is a homotopy pushout. (Technically, 
Mather used a different definition of a ``homotopy commutative cube" as he needed to 
prove his Cube Lemma in greater generality. In our case there is the stronger hypothesis 
that the sides of the cube are homotopy pullbacks over a common base, as they have been 
obtained by composing with the map 
\(\namedright{\Gtau(N)}{t}{N_{0}}\). 
This, together with the fact that the bottom face is a homotopy pushout, lets one use~\cite[Lemma 3.1]{PT}, 
for example, to show that the top face of the cube in~(\ref{Qcube}) is also a homotopy pushout.)

We would like to identify the maps $a$ and $b$ in~(\ref{Qcube}). Let $I=[0,1]$ be the unit interval 
with $0$ as the basepoint.
Recall that the \emph{reduced join} of two pointed spaces $X$ and $Y$ is the quotient space 
\[X\ast Y=(X\times I\times Y)/\sim\] 
where $(x,0,\ast)\sim (x',0,\ast)$, $(\ast,1,y)\sim (\ast,1,y')$ and $(\ast,t,\ast)\sim (\ast,1,\ast)$ 
for all $x,x'\in X$, $y,y'\in Y$ and $t\in I$. It is well known that there is a homotopy equivalence 
$X\ast Y\simeq\Sigma X\wedge Y$. 

\begin{lemma} 
   \label{cubeab} 
   The maps $a$ and $b$ in~(\ref{Qcube}) are homotopic to the projections 
   onto the first and second factor respectively. Consequently, there is a homotopy equivalence 
   $H\simeq F\ast S^{k-1}\simeq \Sigma^{k}F$. 
\end{lemma} 

\begin{proof} 
The rear face of the cube~(\ref{Qcube}) is the iterated homotopy pullback in~(\ref{N0cubea}), 
therefore $a\simeq\pi_{1}$. Next, the map $b$ is induced by the 
homotopy pullback diagram 
\[\diagram 
      F\times S^{k-1}\rto^-{g\times {\rm id}}\dto^{b} & S^{n-1}\times S^{k-1}\rto^-{s\circ ({\rm id}\times \iota)}\dto^{(i\times {\rm id})\circ f} & N_0\ddouble \\  
      S^{k-1}\rto^-{i_{2}} & N_0\times S^{k-1}\rto^-{\pi_{1}} & N_0,
  \enddiagram\] 
where the right square commutes up to homotopy by Lemma \ref{CQ}. Generically, let $\pi_{2}$ 
be the projection onto the second factor. Observe that by definition of $f$, we have $(\pi_2\circ (i\times {\rm id})\circ f)(a, t)=\pi_2(i(\tau(t)a), t)=t=\pi_2(a, t)$, that is, $\pi_2\circ (i\times {\rm id})\circ f=\pi_2$. Thus $\pi_2\circ (i\times {\rm id})\circ f\circ (g\times {\rm id})=\pi_{2}\circ(g\times id)=\pi_2$, and therefore $b=\pi_2\circ i_2\circ b$ is homotopic to $\pi_2$. 

Therefore the homotopy pushout giving $H$ in the top face of~(\ref{Qcube}) 
is equivalent, up to homotopy, to that given by the projections 
\(\namedright{F\times S^{k-1}}{\pi_{1}}{F}\) 
and 
\(\namedright{F\times S^{k-1}}{\pi_{2}}{S^{k-1}}\). 
This homotopy pushout is in turn equivalent, up to homotopy, to the pushout given by the inclusions 
\(\namedright{F\times S^{k-1}}{1\times i}{F\times CS^{k-1}}\) 
and 
\(\namedright{F\times S^{k-1}}{i\times 1}{CF\times S^{k-1}}\), 
where $CS^{k-1}$ and $CF$ are the reduced cones on $S^{k-1}$ and $F$ respectively and the map $i$ 
in both cases is the inclusion into the base of the cone. The latter pushout is, by definition, 
the join $F\ast S^{k-1}$. 
\end{proof} 

Combining these results gives the following. 

\begin{proof}[Proof of Theorem~\ref{gyrationtypeintro}]
Recall from~(\ref{N0cube1}) that the homotopy fibre of $i: S^{n-1}\rightarrow N_0$ is $F$ and 
the homotopy fibre of 
\(\namedright{\Gtau(N)}{t}{N_0}\) 
is $H$. By Lemma~\ref{cubeab}, $H\simeq \Sigma^k F$.
This proves the first assertion. 
By Lemma~\ref{tinverse}, $t$ has a right homotopy inverse. The asserted homotopy 
equivalence for $\Omega \Gtau(N)$ immediately follows.  
\end{proof}  


\section{A loop decomposition of a connected sum}
\label{sec: sum}
Let $S^{k-1}\stackrel{j}{\longrightarrow} E \stackrel{p}{\longrightarrow} M$ be a fibre bundle over a connected closed smooth $n$-manifold $M$ with fibre a connected sphere $S^{k-1}$ (so $k\geq 2$). Let $N$ be a connected closed smooth $n$-manifold. Let $q: N\# M\stackrel{}{\longrightarrow} M$ be the map that pinches $N$ into a point. Taking the pullback of the fibre bundle~$p$ with the map $q$ induces a morphism of fibre bundles
\begin{equation}\label{ENdefeq} 
   \diagram 
       S^{k-1}\rto^-{j_{N}}\ddouble & E_{N}\rto^-{p_{N}}\dto & N\# M\dto^{q} \\ 
       S^{k-1}\rto^-{j} & E\rto^-{p} & M
   \enddiagram 
\end{equation} 
that defines the manifold $E_{N}$ and the maps $j_{N}$ and $p_{N}$.
In this section, we study the homotopy type of $\Omega(N\# M)$ with the help of $\Gtau(N)$ introduced and studied in Section \ref{sec: gyration}. In the sequel a special case will be used in the proof of Theorem \ref{stabledecthm}. We start with a lemma that generalizes a result of Duan~\cite{D} for the case when $k=2$.  

\begin{lemma}\label{ENlemma}
Suppose that $k\in\{2,4,8\}$ and the fibre inclusion $j: S^{k-1}\rightarrow E$ can be extended to an embedding $D^k\hookrightarrow E$. Then there is a diffeomorphism
\[
E_N\cong \Gtau(N)\#E 
\]
for some map $\tau: S^{k-1}\longrightarrow O(n)$.
Further, under this diffeomorphism the fibre inclusion $j_N\colon S^{k-1}\rightarrow E_N$ can be extended to an embedding $D^k\hookrightarrow \Gtau(N)\#E$. 
\end{lemma}

\begin{remark} 
The restriction to $k\in\{2,4,8\}$ is due to the fact that a null homotopy for $j$ in the homotopy fibration 
\(\nameddright{S^{k-1}}{j}{E}{p}{M}\) 
implies that $S^{k-1}$ retracts off $\Omega M$ and is therefore an $H$-space. 
\end{remark} 

\begin{proof}
For a manifold $X$, let  $X_0$ denote $X$ with a small disc removed. In the case of a connected sum, 
observe that there is a homotopy cofibration
\[
N_0\stackrel{j}{\longrightarrow} N\# M\stackrel{q}{\longrightarrow} M,
\]
where $j$ is the canonical inclusion. In particular, the restriction of the bundle $p_N$ in~(\ref{ENdefeq}) to $N_0$ is trivial because it is formed by pulling back 
\(\namedright{E}{p}{M}\) 
and the trivial map $q\circ j$. It is therefore isomorphic as a bundle to the canonical projection $N_0\times S^{k-1}\rightarrow N_0$. Choose any trivialization $N_0\times S^{k-1}\stackrel{\cong}{\longrightarrow} p_N^{-1}(N_0)$. Its restriction to the boundary $\partial(N_{0}\times S^{k-1})=\partial N_{0}\times S^{k-1}=S^{n-1}\times S^{k-1}$ gives an embedding $f_1: S^{n-1}\times S^{k-1}\stackrel{}{\longrightarrow} p_N^{-1}(N_0)$.
On the other hand, the restriction of the bundle $p_N$ to~$M_0$ is obtained by pulling back 
 \(\namedright{E}{p}{M}\) 
 and the inclusion 
 \(\namedright{M_{0}}{}{M}\), 
 so it is isomorphic to $E-F_2(D^n\times S^{k-1})$ for some framing $F_2:D^n\times S^{k-1}\longrightarrow E$ of $S^{k-1}$. Denote the restriction of $F_2$ to the boundary of $D^{n}\times S^{k-1}$ by $f_2: S^{n-1}\times S^{k-1}\stackrel{}{\longrightarrow} p_N^{-1}(M_0)$. 
Then $E_N$ is obtained by gluing $p_N^{-1}(N_0)$ and $p_N^{-1}(M_0)$ together along their boundaries through the diffeomorphism $f_2\circ f_1^{-1}|_{f_1(S^{n-1}\times S^{k-1})}$.

Moreover, since the fibre inclusion $j: S^{k-1}\rightarrow E$ can be extended to an embedding $D^k\hookrightarrow E$, we claim that there exists 
an embedding $\alpha: D^{n+k-1}\hookrightarrow E$ such that ${\rm Im}(F_2)\subseteq \alpha(\accentset{\circ}{D}^{n+k-1})$ up to isotopy. To see this, observe that as the core sphere $S^{k-1}$ of ${\rm Im}(F_2)\cong D^n\times S^{k-1}$ bounds a disk in $E$ we can shrink the disk so that its boundary sphere lies on the torus boundary of ${\rm Im}(F_2)$, as displayed in Figure \ref{torusfigure}. Then, by the fact that
\[
D^{n+k-1}\cong (D^n\times S^{k-1}) \cup (D^{n-1}\times D^k),
\]
we can thicken the shrunken disk $D^k$ to $D^{n-1}\times D^k$ so that its union with ${\rm Im}(F_2)$ is the image of an embedding $\alpha: D^{n+k-1}\hookrightarrow E$ (notice in Figure \ref{torusfigure} that $D^{n-1}$ lives in the vertical dimensions). Finally, we further shrink the embedded image ${\rm Im}(F_2)$ so that it is disjoint from both the boundary sphere $S^{n+k-2}$ of $\alpha(D^{n+k-1})$ and $D^{n-1}\times D^k$ in $E$.

\begin{figure}[!htb]
\centering
\includegraphics[width=2.7in]{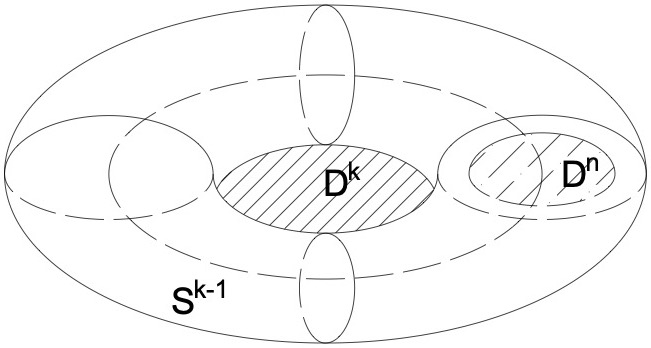}
\caption{The core sphere of the framed $(k-1)$-embedding bounds a disk}\label{torusfigure}
\end{figure}
Hence, there is a decomposition $E\cong S^{n+k-1}\#E$ such that the framing $F_2$ embeds into the $S^{n+k-1}$ factor. By abuse of notation we still denote the new framing and its restriction on the boundary by~$F_2$ and $f_2$ respectively. 

Consider the string of diffeomorphisms  
\begin{equation} 
\label{ENeqns} 
\begin{split}
E_N
& \cong p_N^{-1}(N_0)\cup_{f_2\circ f_1^{-1}} p_N^{-1}(M_0)\\
&\cong (N_0\times S^{k-1}) \cup_{f_2} (E-F_2(D^n\times S^{k-1}))\\
&\cong  (N_0\times S^{k-1}) \cup_{f_2} \big((S^{n+k-1}-(D^n\times S^{k-1})) \# E\big)\\
&\cong  (N_0\times S^{k-1}) \cup_{f_2} \big((S^{n-1}\times D^k) \# E\big) \\ 
& \cong \big((N_0\times S^{k-1})\cup_{f_{2}}(S^{n-1}\times D^{k})\big)\# E \\ 
&\cong  \Gtau(N)\#E. 
\end{split}
\end{equation} 
The first two hold by the identifications of $E_{N}$ and $p_{N}^{-1}(M_{0})$ at the start of the proof, 
the third holds by the decomposition $E\cong S^{n+k-1}\#E$ with the framing $F_2$ embeddng 
into the $S^{n+k-1}$ factor, the fourth holds since $S^{n+k-1}\cong D^{n}\times S^{k-1}\cup S^{n-1}\cup D^{k}$, 
the fifth holds since the construction of $\alpha$ has the property that the boundary ${\rm Im}(f_2)$ 
of ${\rm Im}(F_2)$ is disjoint from the boundary $S^{n+k-2}$ of $E_0$ so the order in which the 
connected sum is taken is irrelevant, and the sixth holds by the definition of 
$\Gtau(N)$ in~(\ref{gyrationdef}), where $\tau: S^{k-1}\rightarrow O(n)$ is determined 
by $f:=f_2^{-1}$ through $f(a, t)=(\tau(t)a,t)$.
This proves the first assertion of the lemma.  

For the second assertion, by the construction above the disk $D^k$ is disjoint from ${\rm Im}(F_2)$ (Figure~\ref{torusfigure}). Hence, it survives under the gluing process described in (\ref{ENeqns}), which gives us an embedding $D^k\hookrightarrow \Gtau(N)\#E$. 
\end{proof}

By the classical Whitney embedding theorem, when $n+k-1\geq 2k+1$, that is, when $n\geq k+2$, any fibre inclusion that is homotopically trivial can be extended to an embedding of the disc. Hence we immediately obtain the following corollary.

\begin{corollary}\label{ENcoro}
Suppose that $k\in\{2,4,8\}$, the fibre inclusion $j: S^{k-1}\rightarrow E$ is null homotopic and $n\geq k+2 \geq 4$. Then there is a diffeomorphism
\[
E_N\cong \Gtau(N)\#E 
\]
for some map $\tau: S^{k-1}\longrightarrow O(n)$. ~$\qqed$
\end{corollary}

We also need a strengthened version of Corollary~\ref{ENcoro} for the special case 
when $k=2$ that was proved by Duan. Recall that the manifold $\Gtau(N)$ is denoted by $\Gzero(N)$ when $\tau$ is trivial. 

\begin{lemma}\cite[Theorem B]{D} \label{duanthmb}
When $k=2$, $n\geq 4$, $E$ is simply connected and $M$ is non-spin, there is a diffeomorphism 
\[
 \hspace{6.1cm} E_N\cong \Gzero(N)\#E. \hspace{6.1cm}\Box
\]
\end{lemma}

\begin{proposition}\label{sumloopdecprop}
Let $N$ and $M$ be two connected closed $n$-dimensional manifolds. Let 
$S^{k-1}\stackrel{j}{\longrightarrow} E \stackrel{p}{\longrightarrow} M$ 
be a fibre bundle such that $k\in\{2,4,8\}$, the fibre inclusion $j$ is null homotopic and $n\geq k+2\geq 4$. Then there is a homotopy equivalence 
\[
\Omega (N\# M)\simeq S^{k-1}\times \Omega (\Gtau(N)\#E)
  \]
for some map $\tau: S^{k-1}\longrightarrow O(n)$.

Alternatively, if $k=2$, $n\geq 4$, $E$ is simply connected and $M$ is non-spin, then there is a homotopy equivalence 
\[
\Omega (N\# M)\simeq S^{1}\times \Omega (\Gzero(N)\#E).
  \]
\end{proposition}
\begin{proof}
Consider the pullback of fibre bundles~(\ref{ENdefeq}). 
By either set of hypotheses, the fibre inclusion 
\(\namedright{S^{k-1}}{j}{E}\) 
is null homotopic so by the Whitney embedding theorem it 
can be extended to an embedding $D^k \hookrightarrow E$.
Then by Lemma \ref{ENlemma} the fibre inclusion $j_N: S^{k-1}\rightarrow E_N$ can be extended to an embedding $D^k\hookrightarrow E_N$. In particular, $j_N$ is null homotopic.  

Now consider the bundle 
\(\nameddright{S^{k-1}}{j_{N}}{E_{N}}{p_{N}}{M\# N}\). 
The null homotopy for $j_{N}$ implies that after looping there is a homotopy equivalence
$\Omega (N\# M)\simeq S^{k-1} \times \Omega E_N$. 
Substituting in the homotopy equivalence for $E_{N}$ in Corollary \ref{ENcoro} or Lemma \ref{duanthmb} in either case then completes the proof.
\end{proof}

\section{Proof of Theorem \ref{stabledecthm} and examples}
\label{sec: ex1}

In this section, we prove Theorem \ref{stabledecthm}, deduce the decompositions 
in Example \ref{ex1intro} (stated in Proposition~\ref{Pnsumprop}), and prove Corollary~\ref{rationalcase}. 

\begin{proof}[Proof of Theorem \ref{stabledecthm}]
Let us first prove the local homotopy equivalences. 
Let $\mathbb{F}=\mathbb{C}$ or $\mathbb{H}$, and take $k=2$ or $4$ correspondingly. 
For $m\geq 2$ in the complex or quaternionic case, let $\mathbb{F}P^{m}$ be the projective space obtained by identifying lines through the origin in $\mathbb{F}^{m}$.  
Recall that~$N$ is a closed $2n$-manifold such that $n$ satisfies the conditions in the theorem for each case. Then $2n=km$ for an integer $m\geq 2$. Consider the standard fibration $S^{k-1}\stackrel{j}{\longrightarrow} S^{km+k-1}\stackrel{p}{\longrightarrow} \mathbb{F}P^m$. 
Since $j$ is null homotopic, by Proposition~\ref{sumloopdecprop} there is a homotopy equivalence
\[
\Omega (N\# \mathbb{F}P^m)\simeq S^{k-1}\times \Omega (\Gtau(N)\#S^{km+k-1})\cong S^{k-1}\times \Omega \Gtau(N)
  \]
for some $\tau: S^{k-1}\longrightarrow O(n)$.
By Theorem \ref{gyrationtypeintro}, after localization away from $\mathcal{P}_k$, there is a homotopy equivalence
\begin{equation}\label{thm1pfeq3}
\Omega\Gtau(N)\simeq \Omega N_0 \times \Omega \Sigma^k F,
\end{equation}
where $F$ is the homotopy fibre of the inclusion
 \(\namedright{S^{2n-1}}{i}{N_0}\) of the boundary.
Combining, we obtain a homotopy equivalence
\[
\Omega (N\# \mathbb{F}P^m)\simeq S^{k-1}\times \Omega N_0 \times \Omega \Sigma^k F
\]
after localization away from $\mathcal{P}_k$. 

We wish to explicitly identify $\mathcal{P}_{k}$. By the definition of $\mathcal{P}_k$ in~(\ref{pkdefeq}) we have $\mathcal{P}_2=\{2\}$ while $\mathcal{P}_{4}$ consists of those primes that divide the denominators of $B_{1}/4$, where $B_{1}$ is the first Bernoulli number. As $B_{1}=6$ we obtain $\mathcal{P}_4=\{2,3 \}$. This completes the proof of the local homotopy equivalences.

To show the integral homotopy equivalence when $k=2$ and $n$ is even, notice that $\mathbb{C}P^n$ is non-spin. Hence by Proposition \ref{sumloopdecprop} there is a homotopy equivalence
\[
\Omega (N\# \mathbb{C}P^n)\simeq S^{1}\times \Omega (\Gzero(N)\#S^{2n+1})\cong S^{1}\times \Omega \Gzero(N).
  \]
By Theorem \ref{gyrationtypeintro} the equivalence (\ref{thm1pfeq3}) holds integrally when $\tau=0$.
Combining, we obtain a homotopy equivalence
\[
\Omega (N\# \mathbb{C}P^n)\simeq S^{1}\times \Omega N_0 \times \Omega \Sigma^2 F.
\]
This completes the proof of the theorem. 
\end{proof}
 
\begin{proposition}\label{Pnsumprop}
For $n\geq 2$, there are homotopy equivalences
\begin{itemize}
\item $\Omega (\mathbb{C}P^{2n}\#\mathbb{C}P^{2n})\simeq S^1\times S^1\times \Omega S^3\times \Omega S^{4n-1}$;
\item $\Omega (\mathbb{C}P^{2n+1}\#\mathbb{C}P^{2n+1})\simeq S^1\times S^1\times \Omega S^3\times \Omega S^{4n+1}$ after localization away from $2$;
\item $\Omega (\mathbb{H}P^{n}\#\mathbb{H}P^{n})\simeq S^3\times S^3\times \Omega S^7\times \Omega S^{4n-1}$ after localization away from $2$ and $3$;
\item $\Omega (\mathbb{C}P^{2n}\#\mathbb{H}P^{n})\simeq S^1\times S^3\times \Omega S^5\times \Omega S^{4n-1}$;
\item $\Omega (\mathbb{C}P^{8}\#\mathbb{O}P^{2})\simeq S^1\times S^7\times \Omega S^9\times \Omega S^{15}$;
\item $\Omega (\mathbb{H}P^{4}\#\mathbb{O}P^{2})\simeq S^3\times S^7\times \Omega S^{11}\times \Omega S^{15}$ after localization away from $2$ and $3$. 
\end{itemize}
\end{proposition}
\begin{proof}
Let us only prove the first three homotopy equivalences after localization away from $\mathcal{P}_{k}$, as the rest can be shown by a similar argument. 
We adopt the notation in the proof of Theorem \ref{stabledecthm}. For $N=\mathbb{F}P^m$, we have $N_0=\mathbb{F}P^{m-1}$ and $F\simeq S^{k-1}$ by the standard fibration $S^{k-1}\stackrel{}{\longrightarrow} S^{km-1} \stackrel{i}{\longrightarrow} \mathbb{F}P^{m-1}$.
Hence by Theorem~\ref{stabledecthm}, after localization away from $\mathcal{P}_{k}$, there is a homotopy equivalence 
\[
\Omega (\mathbb{F}P^m\# \mathbb{F}P^m)\simeq S^{k-1}\times \Omega \mathbb{F}P^{m-1} \times \Omega S^{2k-1}.
\]
By the well known homotopy equivalence $\Omega \mathbb{F}P^{m-1} \simeq S^{k-1}\times \Omega S^{km-1}$, it follows that
\[
\Omega (\mathbb{F}P^m\# \mathbb{F}P^m)\simeq S^{k-1}\times S^{k-1}\times \Omega S^{km-1} \times \Omega S^{2k-1}.
\]
\end{proof}

The proof of Corollary~\ref{rationalcase} requires a preliminary result, which will also be 
used again in Section~\ref{sec: fibattach}. Recall that there is a homotopy cofibration 
\(\nameddright{S^{n-1}}{i}{N_0}{h}{N}\). 

\begin{lemma}\label{fibrehlemma}
If the map 
   \(\namedright{\Omega N_0}{\Omega h}{\Omega N}\) 
   has a right homotopy inverse, then there is a homotopy fibration 
   \[
   \Sigma^{n-1}\Omega N \vee S^{n-1}\longrightarrow N_0\stackrel{h}{\longrightarrow} N, 
   \]
which splits after looping to give a homotopy equivalence
\[
\Omega N_0 \simeq  \Omega N \times \Omega ( \Sigma^{n-1}\Omega N \vee S^{n-1}).
\]
Moreover, the composition $S^{n-1}\stackrel{i_2}{\longrightarrow}\Sigma^{n-1}\Omega N \vee S^{n-1}\longrightarrow N_0$ is homotopic to $i$, where $i_2$ is the inclusion of the second wedge summand.
\end{lemma} 
\begin{proof}
The lemma is an immediate application of~\cite[Proposition 3.5]{BT2} to the homotopy cofibration \(\nameddright{S^{n-1}}{i}{N_0}{h}{N}\). 
\end{proof}

\begin{proof}[Proof of Corollary~\ref{rationalcase}]  
The hypotheses on $N$ allow us to apply~\cite[Th\'{e}or\`{e}me 5.1]{HaL}, which states that the map 
\(\namedright{N_{0}}{h}{N}\) 
induces a surjection on rational homotopy groups. Consequently, as the loop space of any 
simply-connnected rational space is homotopy equivalent to a product of Eilenberg-MacLane 
spaces, the map 
\(\namedright{\Omega N_{0}}{\Omega h}{\Omega N}\) 
has a right homotopy inverse rationally. The existence of this right homotopy inverse lets us 
apply Lemma~\ref{fibrehlemma} in the rational setting, implying that there is a rational homotopy fibration 
\[\nameddright{\Sigma^{2n-1}\Omega N\vee S^{2n-1}}{}{N_{0}}{}{N}\] 
that splits after looping to give a rational homotopy equivalence 
\[\Omega N_{0}\simeq\Omega N\times\Omega(\Sigma^{2n-1}\Omega N\vee S^{2n-1}).\] 

Substituting this into the rational homotopy equivalence 
$\Omega(N\conn\mathbb{C}P^{n})\simeq S^{1}\times\Omega N_{0}\times\Omega\Sigma^{2} F$ 
from Theorem~\ref{stabledecthm} then proves the asserted homotopy equivalence for 
$\Omega(N\conn\mathbb{C}P^{n})$. The homotopy equivalence for 
$\Omega(N\conn\mathbb{H}P^{\frac{n}{2}})$ 
is obtained similarly. 
\end{proof}


\section{The homotopy fibre of certain attaching maps} 
\label{sec: fibattach} 
To take Theorem~\ref{gyrationtypeintro} or Theorem \ref{stabledecthm} further, the space $F$ should be better identified. Consider again the homotopy cofibration 
\(\nameddright{S^{n-1}}{i}{N_0}{h}{N}\). 
By definition, $F$ is the homotopy fibre of the inclusion $i$ of the boundary. Equivalently, it is the homotopy fibre of the attaching map for the top cell of $N$. In this section we study the homotopy type of $F$ under the condition that $\Omega h$ admits a right homotopy inverse, in which case Lemma~\ref{fibrehlemma} holds. 

Assume that the map 
\(\namedright{N_{0}}{h}{N}\) 
has the property that $\Omega h$ has a right homotopy inverse. By Lemma~\ref{fibrehlemma}, 
there is a homotopy fibration 
\(\nameddright{\Sigma^{n-1}\Omega N\vee S^{n-1}}{}{N_{0}}{h}{N}\) 
and the composite 
\(S^{n-1}\stackrel{i_2}{\longrightarrow}\Sigma^{n-1}\Omega N \vee S^{n-1}\longrightarrow N_0\) 
is homotopic to $i$. From the factorization of $i$ we obtain a homotopy fibration diagram 
\begin{equation} 
  \label{fgdgrm} 
  \diagram 
       & \Omega N_0\rdouble\dto^{\partial} &  \Omega N_0\dto^{\Omega h} \\ 
       Z\rto^-{\rho}\ddouble & F\rto^-{\kappa}\dto & \Omega N\dto \\ 
       Z\rto & S^{n-1}\rto^-{i_2}\dto^{i} &\Sigma^{n-1}\Omega N \vee S^{n-1}\dto \\ 
       &  N_0\rdouble & N_0 
  \enddiagram 
\end{equation} 
that defines the space $Z$ and the maps $\kappa$, $\rho$ and $\partial$. 

\begin{lemma} 
   \label{Ftype} 
   If the map 
   \(\namedright{\Omega N_0}{\Omega h}{\Omega N}\) 
   has a right homotopy inverse, then there is a homotopy equivalence 
   $F\simeq\Omega N\times Z$. 
\end{lemma} 

\begin{proof} 
In general, a homotopy fibration 
\(\nameddright{R}{}{S}{}{T}\) 
has a connecting map 
\(\partial\colon\namedright{\Omega T}{}{R}\) 
and a canonical homotopy action  
\[\theta_{R}\colon\namedright{\Omega T\times R}{}{R}\] 
that extends the wedge sum  
\(\namedright{\Omega T\vee R}{}{R}\) 
of $\partial$ and the identity map on $R$. Moreover, this action is natural for maps 
of homotopy fibrations. 

Let 
\(r\colon\namedright{\Omega N}{}{\Omega N_0}\) 
be a right homotopy inverse for $\Omega h$. Consider the diagram 
\begin{equation} 
  \label{actiondgrm} 
  \diagram 
     \Omega N\times Z\rto^-{\pi_{1}}\dto^{r\times \rho} 
           & \Omega N\dto^{i_{1}\circ r} \\ 
     \Omega N_0\times F\rto^-{1\times \kappa}\dto^{\theta_{F}} 
           & \Omega N_0\times\Omega N\dto^{\theta_{\Omega N}} \\ 
     F\rto^-{s} & \Omega N  
  \enddiagram 
\end{equation} 
where $\pi_{1}$  is the projection onto the first factor, $i_{1}$ is the inclusion into the 
first factor, and $\theta_{F}$ and $\theta_{\Omega N}$ are the homotopy actions associated to the 
homotopy fibrations
\(\nameddright{F}{}{S^{n-1}}{i}{N_0}\) 
and 
\(\nameddright{\Omega N}{}{\Sigma^{n-1}\Omega N \vee S^{n-1}}{}{N_0}\) 
respectively. By~(\ref{fgdgrm}), there is a morphism between these two homotopy fibrations 
so the lower square in~(\ref{actiondgrm}) homotopy commutes by the naturality of the 
homotopy action. The upper square in~(\ref{actiondgrm}) homotopy commutes since 
$\kappa\circ \rho$ is the composite of two consecutive maps in a homotopy fibration 
and so is null homotopic. Observe that both squares in~(\ref{actiondgrm}) are in 
fact homotopy pullbacks. Therefore, the outer rectangle is also a homotopy pullback. 
Since the composite $\theta_{\Omega N}\circ i_{1}\circ r$ in the right column of~(\ref{actiondgrm}) 
is homotopic to the identity map, the fact that the outer rectangle in the diagram is a 
homotopy pullback implies that $\theta_{F}\circ(r\times \rho)$ is a homotopy equivalence.  
\end{proof} 

This can be taken further by identifying the homotopy type of $Z$. 

\begin{lemma} 
   \label{Zlemma} 
There is a homotopy equivalence
   \[
   Z\simeq\Omega\big((\Omega S^{n-1}\wedge \Sigma^{n-1}\Omega N)\vee \Sigma^{n-1}\Omega N\big) \simeq \Omega (\bigvee_{\nu=1}^{\infty} \Sigma^{(n-2)\nu+1} \Omega N ).
   \] 
\end{lemma} 

\begin{proof} 
Consider the homotopy fibration diagram 
\[\diagram 
       \Omega V\rto\ddouble & \ast\rto\dto & V\dto \\ 
       \Omega V\rto & S^{n-1}\rto^-{i_2}\dto^{=} &\Sigma^{n-1}\Omega N \vee S^{n-1}\dto^{q_{2}} \\ 
       & S^{n-1}\rdouble & S^{n-1}, 
  \enddiagram\] 
where $q_2$ is the pinch map onto $S^{n-1}$ with homotopy fibre $V$. In particular, the homotopy fibre of~$i_2$ is homotopy 
equivalent to $\Omega V$. Thus $Z\simeq\Omega V$. 
On the other hand, it is well known that the homotopy fibre of the pinch map 
\(\namedright{\Sigma A\vee\Sigma B}{q_{2}}{\Sigma B}\) 
is homotopy equivalent to $\Sigma A\vee(\Sigma A\wedge\Omega\Sigma B)$. 
Further, by the James construction, there is a homotopy equivalence 
$\Sigma\Omega\Sigma B\simeq\bigvee_{\mu=1}^{\infty}\Sigma B^{\wedge\mu}$, 
where $B^{\wedge\mu}$ is the iterated smash product of $\nu$ copies of $B$. Thus there 
is a homotopy equivalence 
\[\Sigma A\vee(\Sigma A\wedge\Omega\Sigma B)\simeq 
      \Sigma A\vee\bigg(\bigvee_{\mu=1}^{\infty}\Sigma A\wedge B^{\wedge\mu}\bigg).\] 
In our case, we obtain that the homotopy fibre of 
\(\namedright{\Sigma^{n-1}\Omega N\vee S^{n-1}}{q_{2}}{S^{n-1}}\) 
is homotopy equivalent to 
\[\Sigma^{n-1}\Omega N\vee\bigg(\bigvee_{\mu=1}^{\infty}\Sigma^{n-1}\Omega N\wedge (S^{n-2})^{\wedge\mu}\bigg) 
     \simeq\bigvee_{\nu=1}^{\infty}\Sigma^{(n-2)\nu+1}\Omega N. \] 
Hence $Z\simeq\Omega(\bigvee_{\nu=1}^{\infty}\Sigma^{(n-2)\nu+1}\Omega N)$.
\end{proof} 

Combining Lemmas~\ref{Ftype} and~\ref{Zlemma} gives the following. 

\begin{proposition} 
   \label{ftype} 
Let \(\nameddright{S^{n-1}}{i}{N_0}{h}{N}\) be a homotopy cofibration such that $\Omega h$ admits a right homotopy inverse. Then there is a homotopy fibration 
   \[\nameddright{\Omega N\times \Omega (\bigvee_{\nu=1}^{\infty} \Sigma^{(n-2)\nu+1} \Omega N )}{}{S^{n-1}}{i}{N_0}.\] 
\end{proposition} 
\vspace{-1cm}~$\qqed$\bigskip 

Specific examples of Proposition~\ref{ftype} will be described in the next section.

\section{More Examples} 
\label{sec: ex2} 
Using Proposition \ref{ftype} in the last section, we refine Theorem~\ref{stabledecthm} in 
the case of $(n-1)$-connected $2n$-dimensional manifolds. For $\ell\geq 1$, let  $J_\ell:=\mathop{\bigvee}\limits_{i=1}^{\ell} S^n$. 

\begin{proposition}\label{n-12nmanifoldprop}
Let $N$ be an $(n-1)$-connected $2n$-dimensional closed manifold with the rank of $H^n(N;\mathbb{Z})$ equal to $d$ for some $d\geq 2$. Then
\begin{itemize} 
\item if $n\geq 6$ is even but $n\neq 8$, there is a homotopy equivalence
\[
\Omega (N\# \mathbb{C}P^n)\simeq S^{1}\times \Omega J_d\times 
      \Omega\Sigma^{2}(\Omega N\times \Omega (\bigvee_{\nu=1}^{\infty} \Sigma^{(2n-2)\nu+1} \Omega N ));
\] 
\item if $n\geq 3$ is odd, there is a homotopy equivalence after localization away from $2$ 
\[
\Omega (N\# \mathbb{C}P^n)\simeq S^{1}\times \Omega J_d\times 
      \Omega\Sigma^{2}(\Omega N\times \Omega (\bigvee_{\nu=1}^{\infty} \Sigma^{(2n-2)\nu+1} \Omega N ));
\]
\item if $n\geq 6$ is even but $n\neq 8$, there is a homotopy equivalence after localization away from $2$ and $3$
\[
\Omega (N\# \mathbb{H}P^{\frac{n}{2}})\simeq S^{3}\times \Omega J_d\times 
      \Omega\Sigma^{4}(\Omega N\times \Omega (\bigvee_{\nu=1}^{\infty} \Sigma^{(2n-2)\nu+1} \Omega N )). 
\]
\end{itemize}
Further, there is a homotopy equivalence
\[\Omega N\simeq\Omega (S^{n}\times S^{n})\times 
                     \Omega\left(J_{d-2}\vee (J_{d-2}\wedge \Omega (S^{n}\times S^{n}))\right).\] 
Consequently: 
   \begin{itemize} 
      \item[(i)] $\Sigma(\Omega N\times \Omega (\mathop{\bigvee}\limits_{\nu=1}^{\infty} \Sigma^{(2n-2)\nu+1} \Omega N ))$ is homotopy equivalent 
                     to a wedge of simply-connected spheres; 
      \item[(ii)] both $\Omega (N\# \mathbb{C}P^n)$ and $\Omega (N\# \mathbb{H}P^{\frac{n}{2}})$ are homotopy equivalent to products of loops on 
                     simply-connected spheres with $S^1$ and $S^3$ respectively.
   \end{itemize}
\end{proposition}
\begin{proof}
Let us only prove the proposition after localization away from $\mathcal{P}_{k}$, while the integral statement can be shown with the same mild modification as in the proof of Theorem \ref{stabledecthm}.
As in the proof of Theorem \ref{stabledecthm}, let $\mathbb{F}=\mathbb{C}$ or $\mathbb{H}$, and take $k=2$, $4$ correspondingly. Then $2n=km$ for $m=1$ or $2$. By Theorem \ref{stabledecthm}, after localization away from $\mathcal{P}_{k}$ there is a homotopy equivalence  
\[
\Omega (N\# \mathbb{F}P^m)\simeq S^{k-1}\times \Omega N_0\times\Omega\Sigma^{k} F,
\]
where $F$ is the homotopy fibre of the inclusion
   \(i: \namedright{S^{2n-1}}{}{N_0}\) of the boundary. Since $N$ is a closed $(n-1)$-connected $2n$-dimensional manifold, there is a homotopy cofibration 
\[\nameddright{S^{2n-1}}{g}{\mathop{\bigvee}\limits_{i=1}^{d} S^{n}}{h}{N}.\]  
Therefore, in this case, $N_0\simeq\mathop{\bigvee}\limits_{i=1}^{d} S^{n}$, that is, $N_{0}\simeq J_{d}$, and $F$ is the homotopy   
fibre of the map $g$. If $d\geq 2$ and $n\notin\{2,4,8\}$ then, using distinct methods, 
in~\cite{BT1,BT2} it was shown that $\Omega h$ has a right homotopy inverse. Therefore 
in these cases Proposition~\ref{ftype} implies that the homotopy fibre of $g$ is 
$\Omega N\times \Omega (\mathop{\bigvee}\limits_{\nu=1}^{\infty} \Sigma^{(2n-2)\nu+1} \Omega N )$. 
We then obtain a homotopy equivalence after localization away from $\mathcal{P}_k$
\[
\Omega (N\# \mathbb{F}P^m)\simeq S^{k-1}\times \Omega J_d\times 
      \Omega\Sigma^{k}(\Omega N\times \Omega (\bigvee_{\nu=1}^{\infty} \Sigma^{(2n-2)\nu+1} \Omega N )).
\] 

Further, it was shown in~\cite{BT1,BT2} that there is a homotopy equivalence 
\[\Omega N\simeq\Omega (S^{n}\times S^{n})\times 
                     \Omega\left(J_{d-2}\vee (J_{d-2}\wedge \Omega (S^{n}\times S^{n}))\right).\] 
This implies that $\Omega N$ is homotopy equivalent to a product of loops on simply-connected spheres. Using the fact that $\Sigma (X\times Y)\simeq \Sigma X\vee \Sigma Y\vee \Sigma(X\wedge Y)$ and the James suspension splitting $\Sigma\Omega\Sigma B\simeq\bigvee_{k=1}^{\infty}\Sigma B^{\wedge k}$, it can be shown that $\Sigma\Omega N$ is homotopy equivalent to a wedge of simply-connected spheres, and hence so is $\mathop{\bigvee}\limits_{\nu=1}^{\infty} \Sigma^{(2n-2)\nu+1} \Omega N$. The Hilton-Milnor Theorem then implies that $\Omega J_d$ and $\Omega (\bigvee_{\nu=1}^{\infty} \Sigma^{(2n-2)\nu+1} \Omega N )$ are homotopy equivalent to products
of loops on simply-connected spheres, and therefore so is $\Omega N\times \Omega (\bigvee_{\nu=1}^{\infty} \Sigma^{(2n-2)\nu+1} \Omega N )$.  
Arguing again as we just have shows that part~(i) holds, and therefore $\Omega\Sigma^{k}(\Omega N\times \Omega (\bigvee_{\nu=1}^{\infty} \Sigma^{(2n-2)\nu+1} \Omega N ))$ is homotopy equivalent to a product of loops on simply-connected spheres, from which part~(ii) follows.
\end{proof} 

\begin{remark} 
Note that the decomposition of $\Omega N$ in Proposition~\ref{n-12nmanifoldprop} implies that its 
homotopy type only depends on $n$ and $d$, and hence the homotopy types of 
$\Omega(N\conn\mathbb{C}P^{n})$ and $\Omega(N\conn\mathbb{H}P^{\frac{n}{2}})$ depend 
only on $n$ and $d$. 
\end{remark}


\bibliographystyle{amsalpha}

\end{document}